\documentclass[12pt, a4paper]{amsart}
\usepackage{amssymb}
\usepackage{tikz} 
\usepackage{mathabx}
\usepackage{graphicx}
\usepackage{mathrsfs}
\usepackage{hyperref}
\usepackage{import}
\usepackage{tcolorbox}

\theoremstyle{plain}

\newtheorem{theorem}{Theorem}[section]
\newtheorem{lemma}[theorem]{Lemma}

\theoremstyle{definition}
\newtheorem{define}{Definition}[section]
\newtheorem{problem}{Problem}
\newtheorem{example}{Example}[section]

\theoremstyle{remark}
\newtheorem{remark}{Remark}[section]

\newtheorem*{acknowledgement}{Acknowledgement}

\begin{document}

\date\today

\title[The boundary behaviour of the squeezing function]{On the boundary behaviour of the squeezing function near linearly convex boundary points}
\author{Ninh Van Thu, Nguyen Thi Lan Huong and Nguyen Quang Dieu\textit{$^{1,2}$}}

\address{Ninh Van Thu}
\address{ School of Applied Mathematics and Informatics, Hanoi University of Science and Technology, No. 1 Dai Co Viet, Hai Ba Trung, Hanoi, Vietnam}
\email{thu.ninhvan@hust.edu.vn}

\address{Nguyen Thi Lan Huong}
\address{Department of Mathematics\\ Hanoi University of Mining and Geology\\ 18 Pho Vien, Bac Tu Liem, Hanoi, Vietnam}
\email{nguyenlanhuong@humg.edu.vn}

\address{Nguyen Quang Dieu}
\address{\textit{$^{1}$}~Department of Mathematics, Hanoi National University of Education, 136 Xuan Thuy, Cau Giay, Hanoi, Vietnam}
 \address{\textit{$^{2}$}~Thang Long Institute of Mathematics and Applied Sciences,
Nghiem Xuan Yem, Hoang Mai, HaNoi, Vietnam}
\email{ngquang.dieu@hnue.edu.vn}



\begin{abstract}
The purpose of this article is twofold. The first aim is to prove that if there exist a sequence $\{\varphi_j\}\subset \mathrm{Aut}(\Omega)$ and  $a\in \Omega$ such that $\lim_{j\to\infty}\varphi_j(a)=\xi_0$ and $\lim_{j\to\infty}\sigma_\Omega(\varphi_j(a))=1$, where $\xi_0$ is a linearly convex boundary point of finite type, then $\xi_0$ must be strongly pseudoconvex. Then, the second aim is to investigate the boundary behaviour of the squeezing function of a general ellipsoid. 
\end{abstract}
\maketitle

\section{introduction}

 Let $\Omega$ be a domain in $\mathbb C^n$ and $p \in \Omega$. For a holomorphic embedding $f\colon \Omega \to \mathbb B^n:=\mathbb B(0;1)$ with $f(p)=0$, we set
$$
\sigma_{\Omega,f}(p):=\sup\left \{r>0\colon B(0;r)\subset f(\Omega)\right\},
$$
where $\mathbb B^n (z;r)\subset\mathbb{C}^n$ denotes the ball of radius $r$ with center at $z$. Then the \textit{squeezing function} $\sigma_{\Omega}: \Omega\to\mathbb R$ is defined as
$$
\sigma_{\Omega}(p):=\sup_{f} \left\{\sigma_{\Omega,f}(p)\right\}.
$$
(See Definition in \cite{DGF12}.) Note that the squeezing function is invariant under biholomorphisms and $0 < \sigma_{\Omega}(z)\leq 1$ for any $z \in \Omega$. Moreover, by definition one sees that $\Omega$ is biholomorphically equivalent to the unit ball $\mathbb B^n$ if $\sigma_{\Omega}(z)=1$ for some $z \in \Omega$.

It is well-known that if $p$ is a strongly pseudoconvex boundary point, then $\lim\limits_{\Omega \ni z\to p\in \partial \Omega}\sigma_{\Omega}(z)=1$ (cf. \cite{DGF16, DFW14, KZ16}). Conversely, motivated by Problem $4.1$ in \cite{FW18}, let us consider the following problem.
\begin{problem}
If $\Omega$ is a bounded pseudoconvex domain with smooth boundary, and if $\lim\limits_{j\to\infty}\sigma_{\Omega}(\eta_j)=1$ for some sequence $\{\eta_j\}\subset \Omega$ converging to $p\in \partial \Omega$, then is the boundary of $\Omega$ strongly pseudoconvex at $p$?
\end{problem}

In the case that $\partial\Omega$ is pseudoconvex of D'Angelo finite type near $\xi_0$, the answer to this problem is affirmative for the following cases:
\begin{itemize}
\item $\{\eta_j\}\subset \Omega$ converges to $\xi_0$ along the inner normal line to $\partial\Omega$ at $\xi_0$ (for details, see \cite{JK18} for $n=2$ and \cite{MV19} for general case).
\item $\{\eta_j\}\subset \Omega$ converges nontangentially to $\xi_0$ (see \cite{Ni18}).
\item $\{\eta_j\}\subset \Omega$ converges $\left(\frac{1}{m_1},\ldots, \frac{1}{m_{n-1}}\right)$-nontangentially to an $h$-extendible boundary point $\xi_0$ (see \cite[Definition $3.4$]{NN19}), where $(1, m_1, \ldots, m_{n-1})$ is the \emph{multitype of $\partial\Omega$ at $\xi_0$} and the \emph{$h$-extendiblility at $\xi_0$} means that the Catlin multitype and D'Angelo multitype of $\partial\Omega$ at $\xi_0$ coincide (see \cite[Definition $3.3$]{Yu95}). 
\end{itemize}

Now we consider the case that $\{\eta_j\}\subset \Omega$ be a sequence converging $\left(\frac{1}{m_1},\ldots, \frac{1}{m_{n-1}}\right)$-nontangentially to $\xi_0$. Then, the condition that $\lim\limits_{j\to \infty} \sigma_\Omega(\eta_j)=1$ ensures that the unit ball $\mathbb B^n$ is biholomorphically equivalent to some model $M_P$ given by
$$
M_P=\left\{z\in \mathbb C^n\colon \mathrm{Re}(z_n)+P(z')<1\right\},
$$
where $P$ is a $\left(\frac{1}{m_1},\ldots, \frac{1}{m_{n-1}}\right)$-homogeneous polynomial on $\mathbb C^{n-1}$ (see \cite[Definition $3.1$]{Yu95}). Therefore, $m_1=m_2=\cdots=m_{n-1}=1$, or $\xi_0$ is strongly pseudoconvex (\cite{NN19}). Unfortunately, the point $\xi_0$ may not be strongly psudoconvex when $\{\eta_j\}\subset \Omega$ does not converge $\left(\frac{1}{m_1},\ldots, \frac{1}{m_{n-1}}\right)$-nontangentially to $\xi_0$. For instance, the following example points out that $\lim\limits_{j\to \infty}\sigma_{\Omega}(\eta_j)=1$ for some $\{\eta_j\}\subset \Omega$ converging to a weakly pseudoconvex boundary point (see also Example \ref{ex2} for general case).
 \begin{example}\label{ex1} Let $E_{1,2}:=\{(z_1,z_2)\in \mathbb C^2\colon |z_2|^2+|z_1|^4<1\}$. Consider the sequence $\displaystyle a_n=\Big(\sqrt[4]{\frac{2}{n}-\frac{2}{n^2}}, 1-\frac{1}{n}\Big)\to (0,1)$ as $n\to \infty$. Denote by $\rho(z):=|z_2|^2-1+|z_1|^4$ a defining function for $E_{1,2}$ and denote by $\sigma(z_1)=|z_1|^4$ a $(\frac{1}{4})$-weighted homogeneous polynomial. Then, a computation shows that 
\begin{align*}
\rho(a_n)=\Big |1-\frac{1}{n}\Big |^2-1+\Big|\sqrt[4]{\frac{2}{n}-\frac{2}{n^2}}\Big  |^4= -\frac{2}{n}+\frac{1}{n^2}+\frac{2}{n}-\frac{2}{n^2}=-\dfrac{1}{n^2}<0
\end{align*}
Therefore, $\displaystyle \mathrm{dist}(a_n,\partial E_{1,2})\approx |\rho(a_n)|=\frac{1}{n^2}$, $\displaystyle |\mathrm{Re}(a_{n2})-1|=\Big|-\frac{1}{n}\Big|=\frac{1}{n}$, and $\displaystyle \sigma(a_{n1})=\sigma(\sqrt[4]{\frac{2}{n}-\frac{2}{n^2}})=\left(\sqrt[4]{\frac{2}{n}-\frac{2}{n^2}}\right)^4=\frac{2}{n}-\frac{2}{n^2}\approx \frac{2}{n}$. This implies that $\{a_n\}$ does not converge $(\frac{1}{4})$-nontangentially to the boundary point $p=(0,1)$.

Let us consider the automorphism $\psi_n\in \mathrm{Aut}(E_{1,2})$, given by
$$
\psi_n(z)=\left( \frac{(1-|a_{n2}|^2)^{1/4}}{(1-\bar{a}_{n2}z_2)^{1/2}} z_1,   \frac{z_2-a_{n2}}{1-\bar{a}_{n2} z_2}\right),
$$
and hence $\psi_n(a_n)=(b_n,0)$, where $b_n=   \dfrac{a_{n1}}{(1-|a_{n2}|^2)^{1/4}}=\dfrac{\sqrt[4]{\frac{2}{n}-\frac{2}{n^2}}}{\sqrt[4]{\frac{2}{n}-\frac{1}{n^2}}}\to 1$ as $n\to \infty$. Since $\psi_n(a_n)$ converges to the strongly pseudoconvex boundary point $(1,0)$ of $\partial E_{1,2}$, it follows that $\sigma_{E_{1,2}}(a_n)=\sigma_{E_{1,2}}(\psi_n(a_n)) \to 1$ as $n\to\infty$. However, the point $(0,1)$ is weakly pseudoconvex. \hfill $\Box$
\end{example}

To give a statement of our result, let us recall that $\partial \Omega$ is {\bf linearly convex near $\xi_0\in\partial \Omega$} if there exists a neighbourhood $U$ of $\xi_0$ such that,  for all $z\in \partial \Omega\cap U,$ 
the intersection
$$
(z+T^{10}_z\partial \Omega)\cap (\Omega\cap U)=\emptyset.
$$
In \cite{Ni09}, the first author proved a characterization of linearly convex domains in $\mathbb C^n$ by their noncompact automorphism groups. 

The first aim of this paper is the following theorem.
\begin{theorem}\label{main thm1}
Let $\Omega$ be a bounded domain in $\mathbb C^n$ with smooth pseudoconvex boundary. Assume that $\xi_0$ is a boundary point of $\Omega$ of D'Angelo finite type such that $\partial \Omega$ is linearly convex at $\xi_{0}$ and there exists a sequence $\{\varphi_j\}\subset \mathrm{Aut}(\Omega)$ such that $\eta_j:=\varphi_j(a)\to\xi_0$ as $j\to \infty$ for some $a\in \Omega$.  If $\lim\limits_{j\to\infty}\sigma_{\Omega}(\eta_j)=1$, then $\partial \Omega$ is strongly pseudoconvex at $\xi_0$.
\end{theorem}
\begin{remark} Thanks to the linear convexity of $\partial \Omega$ near a boundary orbit accumulation point $\xi_0$ and the condition that $\lim\limits_{j\to\infty}\sigma_{\Omega}(\eta_j)=1$, the scaling method can be applied to implies that $\mathbb B^n$ is biholomorphically equivalent to a model $M_P$, where is a real nondegenerate plurisubharmonic polynomial of degree less than or equal to the type of $\partial \Omega$ at $\xi_0$. Moreover, since $\eta_j=\varphi_j(a)$ for some $\{\varphi_j\}\subset \mathrm{Aut}(\Omega)$ and $a\in \Omega$, the scaling method yields $\Omega$ is biholomorphically equivalent to a model $M_P$ (see \cite[Theorem $1.1$]{Ni09}), that is, $\Omega$ is biholomorphically equivalent to the unit ball $\mathbb B^n$. Consequently, the point $\xi_0$ is strongly pseudoconvex, as desired.
\end{remark}

Now we move to the second part of this paper. First of all, let us fix positive integers $m_1,\ldots, m_{n-1}$ and let $P(z')$ be a $(1/m_1,\ldots,1/m_{n-1})$-homogeneous polynomial given by 
\begin{equation*}\label{series expression}
P(z)=\sum_{wt(K)=wt(L)=1/2} a_{KL} {z'}^K  \bar{z}'^L,
\end{equation*}
where $a_{KL}\in \mathbb C$ with $a_{KL}=\bar{a}_{LK}$, satisfying that $P(z')>0$ whenever $z'\ne 0$. Here and in what follows, $z':=(z_1,\ldots,z_{n-1})$ and $ wt(K):=\sum_{j=1}^{n-1} \frac{k_j}{2m_j}$ denotes the weight of any multi-index $K=(k_1,\ldots,k_{n-1})\in \mathbb N^{n-1}$ with respect to $\Lambda:=(1/m_1,\ldots,1/m_{n-1})$. Then the general ellipsoid $D_P$ in $\mathbb C^{n}\;(n\geq1)$, defined in  \cite{NNTK19} by
\begin{equation*}
\begin{split}
D_P &:=\{(z',z_n)\in \mathbb C^{n}\colon |z_n|^2+P(z')<1\}.
\end{split}
\end{equation*}
Throughout this paper, we assume that the domain $D_P$ is a WB-domain, i.e.,  $D_P$ is strongly pseudoconvex at every boundary point outside the set $\{(0',e^{i\theta})\colon \theta\in \mathbb R\}$ (cf. \cite{AGK16}).

For any $s, r\in (0,1]$,  inspired by \cite[Lemma $2.5$]{Liu18}  let us define $D^s_P$ and $D^{s,r}_P$ respectively by
\begin{equation*}
	\begin{split}
D^s_P&:=\{z\in \mathbb C^n \colon |z_n-b|^2+sP(z')<s^2\};\\
D^{s,r}_P&:=\{z\in \mathbb C^n \colon |z_n-b|^2+\dfrac{s}{r}P(z')<s^2\},
\end{split}
\end{equation*}
where $s=1-b$. We note that $D_P^{s,1}=D_P^s$ and the property that $\lim \psi_j^{-1}(D^s_P)=D_P$ for a certain family $\{\psi_j\}\subset \mathrm{Aut}(D_P)$ (cf. Lemma \ref{exhaustion-ellipsoid} in Section \ref{squeezing}) plays a key role in the proofs of our main theorems below. 

Indeed, we prove the following theorem.

\begin{theorem}\label{maintheorem2} Let $\Omega$ be a  subdomain of $D_P$ such that $D^s_P\subset \Omega\subset D_P$ for some $s\in (0,1]$. Then, for any $r\in (0,1)$ there exists $ \gamma_0>0$ depending on $r$ such that 
	$$
	\sigma_\Omega(z)\geq \gamma_0,\; \forall \; z\in D^{s,r}_P\cap \Omega.
	$$ 
\end{theorem}

\begin{remark} Notice that the point $p=(0',1)$ is $(P,s)$-extreme for each domain $\Omega_j$ (cf. \cite[Definition $1.1$]{NNC21} for the notion of $(P,s)$-extreme points) and the convergence of a sequence of points in $D_P^{s,r}$ to $p$ is exactly the $\Lambda$-nontangential convergence introduced in \cite[Definition $3.4$]{NN19}. 
\end{remark}

Now we consider the case that $\{a_j\}\subset \Omega\cap U$ converges $\Lambda$-tangentially to $p=0$ in the sense that for any $0<r<1$ there exists $j_{r}\in \mathbb N$ such that $a_{j}\not \in D_{s,r}$ for all $j\geq j_{r}$. Then, the following theorem shows that the squeezing function converges to $1$ provided  that all $\partial\Omega_j$ share a small neighborhood of the point $(0',1)$ with $\partial D_P$ to which the sequence of points converges $\Lambda$-tangentially. More precisely, we prove the following theorem.

\begin{theorem}\label{maintheorem10} Let $\{\Omega_j\}$ be a sequence of subdomains of $D_P$ such that $\Omega_j\cap U=D_P\cap U$, $j\geq 1$, for a fixed neighborhood $U$ of the origin in $\mathbb C^n$. Let $\{\eta_j\}\subset D_P\cap U$ be a sequence that  converges $\Lambda$-tangentially to $(0',1)$ in $D_P$. Then, $\lim_{j\to\infty}\sigma_{\Omega_j}(\eta_j)=1$.
\end{theorem}

The organization of this paper is as follows: In Sections \ref{technical-section}, we recall a definition and results needed later. Next, in Section \ref{linearly-convex} we give a proof of Theorem \ref{main thm1}. Finally, the proofs of Theorem \ref{maintheorem2} and Theorem \ref{maintheorem10} are given in Section \ref{squeezing}.

\section{Several technical lemmas}\label{technical-section}
First of all, we recall the following definition (see \cite{GK} or \cite{DN09}). 
\begin{define} Let $\{\Omega_i\}_{i=1}^\infty$ be a sequence of open sets in a complex manifold $M$ and $\Omega_0 $ be an open set of $M$. The sequence $\{\Omega_i\}_{i=1}^\infty$ is said to converge to $\Omega_0 $ (written $\lim\Omega_i=\Omega_0$) if and only if 
	\begin{enumerate}
		\item[(i)] For any compact set $K\subset \Omega_0,$ there is an $i_0=i_0(K)$ such that $i\geq i_0$ implies that $K\subset \Omega_i$; and 
		\item[(ii)] If $K$ is a compact set which is contained in $\Omega_i$ for all sufficiently large $i,$ then  $K\subset \Omega_0$.
	\end{enumerate}  
\end{define}

Next, to give proofs of Theorem \ref{maintheorem2} and Theorem \ref{maintheorem10}, we need the following lemma which is a generalization of  \cite[Lemma $2.5$]{Liu18}.
\begin{lemma}\label{exhaustion-ellipsoid} Let $\{\psi_j\}\subset \mathrm{Aut}(D_P)$ be a sequence of automorphisms 
	$$      
	\psi_j(z,w)=\left(\dfrac{\sqrt[2m_1]{1-|a_j|^2}}{\sqrt[m_1]{1+\bar a_j z_n}} z_1,\ldots, \dfrac{\sqrt[2m_{n-1}]{1-|a_j|^2}}{\sqrt[m_{n-1}]{1+\bar a_j z_n}} z_{n-1}, \dfrac{z_n+a_j}{1+\bar a_j z_n}\right),
	$$
where $a_j\in (0,1)$ with $\lim a_j=1$. Then, for any $s\in (0,1)$ we have $\psi_j^{-1}(D_P^s)\to D_P$ as $j\to \infty$. In addition, for any $0<\epsilon<1/2$ and any neighborhood $U$ of $(0',1)$ in $\mathbb C^n$ there exists an integer $j_0\geq 1$  such that $\overline{D_P}\setminus B((0',-1),\epsilon))\subset \psi_j^{-1}(\overline{D_P}\cap U)$ for all $j\geq j_0$.

\end{lemma}

\begin{proof}
	A computation shows that
	\begin{align*}  
		&\hskip 0.5cm\left|\dfrac{z_n+a_j}{1+a_jz_n}-b\right|^2+ s P\left(\dfrac{\sqrt[2m_1]{1-|a_j|^2}}{\sqrt[m_1]{1+\bar a_j w}} z_1,\ldots, \dfrac{\sqrt[2m_{n-1}]{1-|a_j|^2}}{\sqrt[m_{n-1}]{1+\bar a_j z_n}} z_{n-1} \right) <s^2\\
		&\Leftrightarrow \left|\dfrac{z_n+a_j}{1+a_jz_n}-b\right|^2+ s \dfrac{1-|a_j|^2}{|1+a_jz_n|^2}P(z)    <s^2\\
		&\Leftrightarrow \left|w-\frac{b(1-a_j)}{1+a_j-2a_jb}\right|^2+ \dfrac{(1-b)(1+a_j)}{1+a_j-2a_jb} P(z)< \dfrac{1+a_j-2b}{1+a_j-2a_jb}+\left|\frac{b(1-a_j)}{1+a_j-2a_jb}\right|^2.
	\end{align*}    
	Moreover, by a straightforward calculation, one has that
	\begin{align*}
		\lim\limits_{j\to\infty} \frac{b(1-a_j)}{1+a_j-2a_jb}=0,\; \lim\limits_{j\to\infty}  \dfrac{(1-b)(1+a_j)}{1+a_j-2a_jb}=1,\; \lim\limits_{j\to\infty}  \dfrac{1+a_j-2b}{1+a_j-2a_jb}=1.
	\end{align*}
	This yields $\psi_j^{-1}(D_P^s)\to D_P$ as $j\to \infty$.
\end{proof}

\section{Squeezing function for linearly convex domains}\label{linearly-convex}

Throughout this section, the domain $\Omega\subset \mathbb C^n$ and the boundary point $\xi_0\in \partial \Omega $ are assumed to satisfy the hypothesis of Theorem \ref{main thm1}, namely $\partial \Omega $ is linearly convex,  of finite type $2m$ near a point $\xi_0$ of $\partial \Omega$. We may also assume that $\xi_0=0$. There exists a neighbourhood $U$ of $\xi_0=0 $ in $\mathbb C^{n}$ such that $\Omega \cap U$ is linearly convex and is defined by a smooth function 
$$
\rho(z',z_n)=\mathrm{Re}(z_n)+h\big(\mathrm{Im}(z_n),z'\big),
$$
where $h $ is a function of class $C^\infty$. We may also assume that there exists a real positive number $\epsilon_0$ such that for every $-\epsilon_0<\epsilon<\epsilon_0$, the level sets $\{\rho(z)=\epsilon\}$ are linearly convex. 

For  each $\epsilon\in (0,\epsilon_0/2), \ \eta\in \Omega\cap U$ with $|\rho(\eta)|<\epsilon_0/2$ and each unit vector $v\in \mathbb S^{n-1}:=\{v\in \mathbb C^n\colon |v|=1\},$ we set
$$\tau(\eta,v, \epsilon):=\sup\{r>0:\rho(\eta+\lambda v)-\rho(\eta)<\epsilon\ \text{for all } \  \lambda\in \mathbb C \ \text{with}\ |\lambda|<r\}. $$
Then, it is easy to see that $\tau(\eta,v,\epsilon)$ is the distance from $\eta$ to $S_{\eta,\epsilon}:=\{\rho(z)=\rho(\eta)+\epsilon\}$ along the complex line $\{\eta+\lambda v\colon  \lambda\in \mathbb C\}$.

To every point $\eta\in \Omega\cap U $ and every sufficiently small positive constant $\epsilon$ we associate
\begin{enumerate}
\item[(1)] A holomorphic coordinate system $(z_1,z_2,\ldots,z_n)$ centered at $\eta$ and preserving orthogonality,  
\item[(2)] Points $p_1,p_2,\ldots, p_n$ on the hypersurface $S_{\eta,\epsilon}$ and,
\item[(2)] Positive  real numbers $\tau_1(\eta,\epsilon),\tau_2(\eta,\epsilon),\ldots,\tau_n(\eta,\epsilon)$.
\end{enumerate}
The construction proceeds as follows. We first set 
$$
e_n:=\frac{\nabla \rho(\eta)}{|\nabla \rho(\eta)|}\ \text{and} \ \tau_n(\eta,\epsilon):=\tau (\eta,e_n,\epsilon).
$$
Working with sufficiently small $\epsilon$, there exists a unique point $p_n$ in $S_{\eta,\epsilon}$ where this distance is achieved. Choose a parameterization of the complex line from $\eta$ to $p_n$ such that  $z_n(0)=\eta$ and $p_n$ lies on the positive $\mathrm{Re}(z_n)$ axis. By the choice of real axis for $z_n$, we have $\frac{\partial r}{\partial x_n}(\eta)=1$ and thus, if $U$ is small enough,
    $$\frac{\partial r}{\partial x_n}(z)\approx 1 \ \text{ for  all } \ z\in U.$$

 We also have 
\begin{equation}\label{eq}
\tau_n(\eta,\epsilon)\approx \epsilon,
\end{equation}
where the constant is independent of $\eta$ and $\epsilon$. Now consider the orthogonal complement $H_n$ of the span of the coordinate $z_n$ in $\mathbb C^n.$  For any $\gamma\in H_n\cap \mathbb S^{n-1}$, compute $\tau(\eta,\gamma,\epsilon)$. Because of the assumption of finite type, the largest such distance is finite and is achieved at a vector $e_1\in H_1\cap \mathbb S^{n-1} $. Set $\tau_1(\eta,\epsilon):= \tau(\eta,e_1,\epsilon).$ Let $p_1\in S_{\eta,\epsilon}$ be a point such that $p_1=\eta+\tau_1(\eta,\epsilon) e_1$. The coordinate $z_1$ is defined by parameterizing the complex line from $\eta$ to $p_1$ in such a  way that $z_1(0)=\eta$ and $p_1$ lies on the positive $\mathrm{Re}(z_1)$ axis. For the next step, define $H_1$ as the orthogonal complement of the span of $z_1$ and $z_n$ and repeat the above construction. Continuing this process, we obtain $n$ coordinate functions $z_k$, vectors $e_k$,  the numbers $\tau_k(\eta,\epsilon)$ and the distinguished points $p_k \ (1\leq k\leq n)$. Let $z_k=x_k+iy_k\ (1\leq k\leq n)$ denote the underlying real coordinates. 

We assume that $\xi_0$ is an accumulating point for a sequence of automorphisms of $\Omega$. Let $\{\eta_j\}\subset \Omega$ be a sequence converging to $\xi_0$. Moreover, we may assume that $\eta_j\in \Omega\cap U$ for all $j$. Let us set $\epsilon_j:=-\rho(\eta_j)$ for all $j$. Then, by argument as above, we construct the new coordinates $(z^j_1,\ldots , z^j_n)$, the positive numbers $\tau_{j,1},\ldots,\tau_{j,n}$, and the points $p_1^j,\ldots,p_n^j$ associated with $\eta_j$ and $\epsilon_j$.

The change of coordinates from the canonical system to the system $(z^j_1,\ldots , z^j_n)$ is the composition of a translation $T_j$ and of a unitary transform $A_j$. In addition, we may assume that $(A_j\circ T_j)^{-1}$ is defined in a fixed neighborhood of the origin and thus the corresponding defining function $\rho_j$ is defined by
$$
\rho_j:= \rho\circ ( A_j\circ T_j)^{-1},
$$
which is given in a fixed neighborhood of $0$ by
\begin{equation*}
\begin{split}
\rho_j(z)=-\epsilon_j+ \mathrm{Re}(\sum_{k=1}^na_k^j z_k)&+\sum_{2\leq |\alpha|+|\beta|\leq 2m} C^j_{\alpha\beta} {z'}^\alpha {z'}^\beta+O(|z|^{2m+1}),
\end{split}
\end{equation*}
where $\alpha=(\alpha_1,\ldots ,\alpha_{n-1}), \ |\alpha|= \alpha_1+\cdots+\alpha_{n-1}$ and ${z'}^\alpha=z_1^{\alpha_1}.\ldots.z_{n-1}^{\alpha_{n-1}}$. We note that $O(|z|^{2m+1})$ is independent of $j$.

Let $\rho\circ A$ be the limit of $\rho_j$ when $j$ goes to infinity, where $A$ is a unitary transform and this convergence is $\mathcal C^\infty $ on a fixed compact neighborhood of $\xi_0$. Then, for every $j $ less than or equal to $n$ and for every multi-index $\alpha$ and $\beta$ satisfying $2\leq |\alpha|+|\beta|\leq 2m$, there exist two complex numbers $a_j$ and $C_{\alpha\beta}$ such that 
$$
\lim_{j\to\infty} a^j_k=a_k\ \text{and}\ \lim_{j\to\infty} C^j_{\alpha\beta}=C_{\alpha\beta}.
$$
Now let us consider the dilation 
$$\Lambda_j(z):=(\tau_{j,1}z_1,\ldots,\tau_{j,n}z_n)$$     
and the function
$$\tilde \rho_j=\frac{1}{\epsilon_j}\rho_j\circ \Lambda_j.$$
Therefore, the defining function $\tilde \rho_j$ has the following form
\begin{equation*}
\begin{split}
\tilde \rho_j(z)=-1 +\frac{1}{\epsilon_j} \mathrm{Re} \big (\sum_{j=1}^n  a_j^j \tau_{j,j}z_j\big )&+\frac{1}{\epsilon_j}\sum_{2\leq |\alpha|+|\beta|\leq 2m} C^j_{\alpha\beta}\tau_j^{\alpha+\beta} {z'}^\alpha {z'}^\beta+O((\epsilon_j)^{1/2m}|z|^{2m+1}),
\end{split}
\end{equation*}
where $\tau_j^{\alpha+\beta}=\tau_{j,1}^{\alpha_1+\beta_1}.\ldots.\tau_{j,n-1}^{\alpha_{n-1}+\beta_{n-1}}$.
Furthermore, it follows from \cite[Prop.3.1]{Ni09} that the functions $\tilde \rho_j$ are smooth and plurisubharmonic, and after taking a subsequence, we may assume that $\{\tilde \rho_j\}$ that converges uniformly on compacta of $\mathbb C^n$ to a smooth plurisubharmonic function $\tilde \rho$ of the form
$$
\tilde \rho(z)=-1+ \mathrm{Re} \big(\sum_{k=1}^n b_k z_k\big)+P(z'),
$$
where $P$ is a plurisubharmonic polynomial of degree less than or equal to $2m$. 

In what follows, let us denote by $\Gamma_j:=\Lambda^{-1}_j\circ A_j\circ T_j$ for all $j$. Then, one can deduce that $\{\Gamma_j(\Omega\cap U)\}$ converges to the following model 
$$
\widetilde{M}_P:=\left\{z\in \mathbb C^n\colon \tilde \rho(z)=-1+ \mathrm{Re} \big(\sum_{k=1}^n b_k z_k\big)+P(z')<0\right\}, 
$$
which is clearly biholomorphically equivalent to 
$$
M_P:=\left\{z\in \mathbb C^n\colon \hat \rho(z):=\mathrm{Re}(z_n)+P(z')<0\right\}.
$$

Let  us consider a sequence of the biholomorphisms $F_j \colon f_j(\Omega\cap U) \to \Gamma_j(\Omega\cap U)$ defined by $F_j=\Gamma_j\circ f_j^{-1}$. Since $F_j(0)=0\in \widetilde{M}_P$, it follows that our sequence $\{F_j\}$ is not compactly divergence. Moreover, the normality of $\{F_j\}$ is ensured by \cite[Lemma 4.1]{Ni09}.

Now we are ready to prove Theorem \ref{main thm1}.
\begin{proof}[Proof of Theorem \ref{main thm1}]
Let $\{\eta_j\}\subset\Omega$ be a sequence given in Theorem \ref{main thm1}, that is, $\lim\limits_{j\to\infty}\eta_j=\xi_0$ and $\lim\limits_{j\to\infty}\sigma_{\Omega}(\eta_j)=1$. Firstly, let us set $\delta _j=2(1-\sigma_{\Omega}(\eta_j))$ for all $j$. Then by our assumption, for each $j$, there exists an injective holomorphic map $f_j:\Omega\to \mathbb{B}^n $ such that $f_j(\eta_j)=(0',0)$ and $\mathbb{B}(0;1-\delta _j)\subset f_j(\Omega)$. Then by \cite[Proposition $2.2$]{DN09} and the hypothesis of Theorem \ref{main thm1}, without loss of generality we may assume that for each compact subset $K\Subset \mathbb B^n$ and each neighborhood $U$ of $\xi_0$, there exists an integer $j_0$ such that $f_j^{-1}(K)\subset \Omega\cap U$ for all $j\geq j_0$, i.e. $f_j(\Omega \cap U)$ converges to $\mathbb B^n$.

Next, it follows from  \cite[Lemma 4.1]{Ni09} that the sequence $\Gamma_j\circ f_j^{-1} \colon f_j(\Omega \cap U) \to  \Gamma_j(\Omega \cap U) $ is normal and its limit is a holomorphic mapping from $\mathbb B^n$ to $\widetilde M_P$. Moreover, by Montel's theorem the sequence $ f_j\circ \Gamma_j^{-1} \colon \Gamma_j(\Omega \cap U)\to   f_j(\Omega \cap U) \subset \mathbb B^n$ is also normal. In addition, our the sequence $\{\Gamma_j\circ f_j^{-1}\}$ is not compactly divergent since $\Gamma_j\circ f_j^{-1}(0,0')=(0,0')$. Then by \cite[Proposition $2.1$]{DN09}, after taking some subsequence of $\{\Gamma_j\circ f^{-1}_{j}\}$, we may assume that such a subsequence converges uniformly on every compact subset of $\mathbb B^n$ to a biholomorphism $F$ from $\mathbb B^n$ onto $\widetilde M_P$, which is clearly equivalent to $M_P$. 

On the other hand, by \cite[Theorem $1.1$]{Ni09} $\Omega$ is also biholomorphically equivalent to $M_P$, and hence $\Omega$ is biholomorphically equivalent to $\mathbb B^n$. Therefore, $\partial\Omega$ is strongly pseudoconvex at $\xi_0$ ($\xi_0$ is of the D'Angelo type $2$), which ends our proof. 
\end{proof}

\section{Proofs of  Theorem \ref{maintheorem2} and Theorem \ref{maintheorem10}}\label{squeezing}

This section is devoted to proofs of Theorem \ref{maintheorem2} and Theorem \ref{maintheorem10}.

\begin{proof}[Proof of Theorem \ref{maintheorem2}]
Let $q=(q',q_n)\in D^{s,r}_P$ near $p=(0',1)$. By the invariance of $D_{s,r}, D_P$ under the rotation $(z',z_n)\mapsto (z',e^{i\theta}z_n)$ for $\theta\in \mathbb R$ satisfying that $\mathrm{Im}(e^{i\theta}  q_n)=0$, without loss of generality we may assume that $\mathrm{Im}(q_{n})=0$ for every $j\geq 1$. 
	
	We now consider the automorphism $\psi_a \in \mathrm{Aut}(D_P)$, given by
	$$	
	\psi_a(z)=\left(\dfrac{\sqrt[2m_1]{1-|a|^2}}{\sqrt[m_1]{1+\bar a z_n}} z_1,\ldots, \dfrac{\sqrt[2m_{n-1}]{1-|a|^2}}{\sqrt[m_{n-1}]{1+\bar a z_n}} z_{n-1}, \dfrac{z_n+a}{1+\bar a z_n} \right),
	$$
	where $a=\mathrm{Re}(q_{n})=q_{n}\in (0,1)$. Then, Lemma \ref{exhaustion-ellipsoid} yields
	\begin{equation*}
		\begin{split}
			\lim\limits_{a\to 1}\psi_a^{-1}(D_{s,r})= D_{P,r} ; \; \lim\limits_{a\to 1}\psi_a^{-1}(\Omega)= D_P,
		\end{split}
	\end{equation*}
	where $\displaystyle D_{P,r}:=D_{P/r}=\Big\{z\in \mathbb C^{n}\colon |z_n|^2+ \frac{1}{r} P(z')<1\Big\}$. Moreover, we have that $\psi_a^{-1}(q)=\Big(\dfrac{q_{1}}{\lambda^{1/2m_1}},\ldots, \dfrac{q_{n-1}}{\lambda^{1/2m_{n-1}}} ,0\Big)\in D_{P,r}\cap \{z_n=0\}$, where $\lambda=1-|a|^2$ and $D_{P,r}\cap \{z_n=0\}\Subset D_P\cap \{z_n=0\}$. Therefore, by Lemma $2.1$ in \cite{NNC21} and again by the invariance of the squeezing function under biholomorphisms, we conclude that 
	$$
	\sigma_\Omega(q)=\sigma_{\psi_a^{-1}(\Omega)}(\psi_a^{-1}(q))>\delta/d>0, \;\forall\; q\in E^r\cap B(0;\epsilon_0),
	$$
	where $d$ denotes the diameter of $D_P$ and $\delta:=\mathrm{dist}(Z_\rho(P), Z_1(P))/2$ with $Z_\rho(P)=\{z'\in \mathbb C^{n-1}\colon P(z')=r\}$.
\end{proof}

\begin{proof}[Proof of Theorem \ref{maintheorem2}]
	For each $j\geq1$, choose $\theta_j\in \mathbb R$ such that that $\mathrm{Im}(e^{i\theta_j}\eta_{jn})=0$. Since $\mathrm{Im}(\eta_{jn})\to 0$ as $j\to\infty$, one has that $\theta_j\to 0$ as $j\to \infty$. Moreover, by shrinking $U$ if necessary we may also assume that $\rho_{\theta_j}(\Omega_j)\cap U=D_P\cap U$ for all $j\geq 1$. Therefore, by the invariance of $D_P$ under the rotation $\rho_{\theta_j}\colon (z',z_n)\mapsto (z',e^{i\theta_j}z_n)$ for $j\geq 1$, without loss of generality we may assume that $\mathrm{Im}(\eta_{jn})=0$ for every $j\geq 1$ and $\{\Omega_j\}$ is a sequence of subdomains of $D_P$ such that $\Omega_j\cap U=D_P\cap U$ for all $j\geq 1$.

We now consider the sequence of automorphisms $\{\psi_j\}\subset \mathrm{Aut}(D_P)$, given by
$$      
\psi_j(z)=\left(\dfrac{\sqrt[2m_1]{1-|a_j|^2}}{\sqrt[m_1]{1+\bar a_j z_n}} z_1,\ldots, \dfrac{\sqrt[2m_{n-1}]{1-|a_j|^2}}{\sqrt[m_{n-1}]{1+\bar a_j z_n}} z_{n-1}, \dfrac{z_n+a_j}{1+\bar a_j z_n} \right),
$$
where $a_j=\mathrm{Re}(\eta_{jn})=\eta_{jn}\in\mathbb R$ for all $j\geq 1$.

Let us set $b_j=\psi_j^{-1}(\eta_j)$ for all $j\geq 1$. Then, a straightforward computation shows that 
$$
b_j=\psi_j^{-1}(\eta_j)=\Big(\dfrac{\eta_{j1}}{\lambda_j^{1/2m_1}},\ldots, \dfrac{\eta_{j(n-1)}}{\lambda_j^{1/2m_{n-1}}} ,0\Big)\in D_{P}\cap \{z_n=0\}, 
$$
where $\lambda_j=1-|a_j|^2$ for all $j\geq 1$.

Since $\{\eta_j\}$ converges $\Lambda$-tangentially to $(0',1)$, it follows that there exists a sequence $\{\rho_j\}\subset (0,1)$ with $\rho_j\to 1$ as $j\to\infty$ such that
$$
|\eta_{jn}-1-s|^2+\dfrac{s}{\rho_j}P(\eta_j')>s^2, \;\forall j\geq 1.
$$
This implies that
\begin{equation*}
	\begin{split}
		P(b_j')&=\dfrac{1}{\lambda_j} P(\eta_j')\geq \dfrac{2\rho_j(1-a_j)}{1-a_j^2}-\dfrac{\rho_j}{s}\dfrac{|1-a_j|^2}{1-a_j^2}\\
		&\geq \dfrac{2\rho_j}{1+a_j}-\dfrac{\rho_j}{s}\dfrac{(1-a_j)}{1+a_j}
	\end{split}
\end{equation*}
for all $j\geq 1$. Therefore, we obtain that $P(b_j')\to 1$ as $j\to \infty$, and hence by passing to a subsequence if necessary, we may assume that $\psi_j^{-1}(\eta_j)$ converges to some strongly pseudoconvex  boundary point $p\in \partial D_P\cap\{z_n=0\}$. Since $a_j\to 1$ as $j\to\infty$, Lemma \ref{exhaustion-ellipsoid} yields
\begin{equation*}
	\begin{split}
		\lim\limits_{j\to\infty}\psi_j^{-1}(\Omega_j) =\lim\limits_{j\to\infty}\psi_j^{-1}(\Omega_j\cap U)=\lim\limits_{j\to\infty}\psi_j^{-1}(D_P\cap U)= D_P.
	\end{split}
\end{equation*}
In addition, for any $\epsilon>0$ sufficiently small there exists $j_0\geq 1$ such that
 $$
 \psi_j^{-1} (\overline{\Omega_j})\setminus B((0',-1), \epsilon)=\overline{D_P}\setminus B((0',-1), \epsilon) 
 $$
  for any $j\geq j_0$. Hence, since $\sigma_{D_P}(b_j)\to 1$ as $j\to \infty$ and by Theorem $3.1$ in \cite{KZ16}, one concludes that $\sigma_{\Omega_j}(\eta_j)=\sigma_{\psi_j^{-1}(\Omega_j)}(b_j) \to 1$ as $j\to \infty$.
\end{proof}

We close this section with an example, which is a generalization of Example \ref{ex1}.
\begin{example}\label{ex2}
Fix positive integers $m_1,\ldots, m_{n-1}$ and denote by $\Lambda:=(1/m_1,\ldots,1/m_{n-1})$. Let us consider a general ellipsoid $D_P$ in $\mathbb C^n\;(n\geq2)$, defined  by
\begin{equation*}
\begin{split}
D_P :=\{(z',z_n)\in \mathbb C^n\colon |z_n|^2+P(z')<1\},
\end{split}
\end{equation*}
where $P(z')$ is a $(1/m_1,\ldots,1/m_{n-1})$-homogeneous polynomial given by 
\begin{equation*}\label{series expression}
P(z')=\sum_{wt(K)=wt(L)=1/2} a_{KL} {z'}^K  \bar{z'}^L,
\end{equation*}
where $a_{KL}\in \mathbb C$ with $a_{KL}=\bar{a}_{LK}$, satisfying that $P(z')>0$ whenever $z'\ne 0$. Suppose that the domain $D_P$ is a WB-domain, i.e.,  $\partial D_P$ is strongly pseudoconvex at every boundary point outside the set $\{(0',e^{i\theta})\colon \theta\in \mathbb R\}$ (cf. \cite{AGK16}).

 Now let us denote by $\rho(z):=|z_n|^2-1+P(z')$ a local defining function for $D_P$ and consider a sequence $\{a_j=(a_j',a_{jn})\}\subset D_P$ which converges $\Lambda$-tangentially to $p:=(0',1)$. Since $D_P$ is invariant under the map $z' \mapsto z'; z_n\mapsto e^{i\theta} z_n$ and $\sigma_{D_P}$ is invariant under biholomorphisms, we may assume that $\mathrm{Im}(a_{jn})=0$ for all $j$. Since $\mathrm{dist}(a_j, \partial D_P)\approx -\rho(a_j)\approx 1-|a_{jn}|^2-P(a_j')$ and $\{a_j\}$ converges $\Lambda$-tangentially to $p$, it follows that $P(a_j')\geq \; c_j \; \mathrm{dist}(a_j, \partial D_P)$ for some sequence $\{c_j\}\subset \mathbb R$ with $0<c_j\to +\infty$. This implies that $P(a_j')\geq c_j'(1-|a_{jn}|^2-P(a_j'))$ for some sequence $\{c_j'\}\subset \mathbb R$ with $0<c_j'\to +\infty$ and hence
$$
P(a_j')\geq \dfrac{c_j'}{1+c_j'}(1-|a_{jn}|^2), \forall j\geq 1.
$$

Let us denote by $\psi_j$ the automorphism of $D_P$, given by
$$
\psi_j(z)=\left( \frac{(1-|a_{jn}|^2)^{1/2m_1}}{(1-\bar{a}_{jn}z_n)^{1/m_1}} z_1,\ldots,  \frac{(1-|a_{jn}|^2)^{1/2m_{n-1}}}{(1-\bar{a}_{jn}z_n)^{1/m_{n-1}}} z_{n-1}, \frac{z_n-a_{jn}}{1-\bar{a}_{jn} z_n}\right),
$$
and hence $\psi_j(a_j)=(b_j,0)$, where 
$$
b_j=  \left( \frac{a_{j1}}{(1-|a_{jn}|^2)^{1/2m_1}} ,\ldots, \frac{a_{j (n-1)}}{(1-|a_{jn}|^2)^{1/2m_{n-1}}}\right). 
$$
Thanks to the boundedness of  $\{b_j\}$, without loss of generality we may assume that $b_j\to b\in \mathbb C^{n-1}$ as $j\to \infty$. In addition, we have that $P(b_j)=\dfrac{1}{1-|a_{jn}|^2}P(a_j')\geq \dfrac{c_j'}{1+c_j'}, \forall j\geq 1$.
Therefore, we arrive at the situation  $b_j\to b$ with $P(b)=1$ and thus $\psi_j(a_j)$ converges to the strongly pseudoconvex boundary point $(b,0)$ of $\partial D_P$, which implies that $\sigma_{D_P}(a_j)=\sigma_{D_P}(\psi_j(a_j)) \to 1$ as $j\to \infty$ even the boundary point $p$ is weakly pseudoconvex. 
\end{example}

\begin{acknowledgement}Part of this work was done while the first author was visiting the Vietnam Institute for Advanced Study in Mathematics (VIASM). He would like to thank the VIASM for financial support and hospitality. 
\end{acknowledgement}

\bibliographystyle{plain}

\begin{thebibliography}{99}
\bibitem[AGK16]{AGK16} T. Ahn, H. Gaussier, and K.-T. Kim, Positivity and completeness of invariant metrics, J. Geom. Anal.
{\bf 26} (2) (2016), 1173--1185.












\bibitem[DGZ12]{DGF12} F. Deng, A. Guan, L. Zhang, Some properties of squeezing functions on bounded domains, Pacific J. Math. {\bf 257} (2)  (2012), 319--341. 
\bibitem[DGZ16]{DGF16} F. Deng, A. Guan, L. Zhang, Properties of squeezing functions and global transformations of bounded domains, Trans. Amer. Math. Soc.  {\bf 368} (4) (2016), 2679--2696. 
\bibitem[DFW14]{DFW14} K. Diederich,  J.E. Forn{\ae}ss,  E. F. Wold, Exposing points on the boundary of a strictly pseudoconvex or a locally convexifiable domain of finite $1$-type,  J. Geom. Anal. {\bf 24} (2014), 2124--2134.
\bibitem[DN09]{DN09}  Do Duc Thai, Ninh Van Thu, Characterization of domains in $\mathbb  C^n$ by their noncompact automorphism groups, Nagoya Math. J. {\bf 196} (2009), 135--160.
\bibitem[FW18]{FW18} J.E. Forn{\ae}ss, E.F. Wold, A non-strictly pseudoconvex domain for which the squeezing function tends to $1$ towards the boundary, Pacific J. Math. {\bf 297}(1) (2018), 79--86.
\bibitem[GK87]{GK} R.E. Greene, S.G. Krantz, Biholomorphic self-maps of domains, Lecture Notes in Math.,  {\bf 1276} (1987), 136--207.

\bibitem[JK18]{JK18} S. Joo,  K.-T. Kim, On boundary points at which the squeezing function tends to one, J. Geom. Anal. {\bf 28} (3)  (2018), 2456--2465.
\bibitem[KZ16]{KZ16} K.-T. Kim, L. Zhang, On the uniform squeezing property and the squeezing function, Pac. J. Math.  {\bf 282} (2) (2016), 341--358.
\bibitem[MV19]{MV19}  P. Mahajan, K. Verma, A comparison of two biholomorphic invariants, Internat. J. Math. {\bf 30} (1) (2019), 1950012, 16 pp.
\bibitem[Liu18]{Liu18} B. Liu, Two applications of the Schwarz lemma, Pacific J. Math. {\bf 296} (2018), no. 1, 141--153.
\bibitem[Nik18]{Ni18} N. Nikolov, Behavior of the squeezing function near $h$-extendible boundary points, Proc. Amer. Math. Soc. {\bf 146} (8) (2018), 3455--3457.
\bibitem[NN20]{NN19} Ninh Van Thu and Nguyen Quang Dieu, Some properties of $h$-extendible domains in $\mathbb C^{n+1}$, J. Math. Anal. Appl. {\bf 485} (2020), no. 2, 123810, 14 pp.. 
\bibitem[NNC21]{NNC21}Ninh Van Thu, Nguyen Thi Kim Son an Chu Van Tiep, Boundary behaviour of the squeezing function near a global extreme point, Complex Variables and Elliptic Equations,  https://doi.org/10.1080/17476933.2021.1991330.
\bibitem[NNTK19]{NNTK19} Ninh Van Thu, Nguyen Thi Lan Huong, Tran Quang Hung, and Hyeseon Kim, On the automorphism groups of finite multitype models in $\mathbb C^{n}$, J. Geom. Anal. {\bf 29} (2019), no. 1, 428--450.
\bibitem[Ni09]{Ni09} Ninh Van Thu, Characterization of linearly convex domains in $\mathbb C^n$ by their noncompact automorphism groups, Vietnam J. Math. 37 (2009), no. 1, 67--79.
\bibitem[Yu95]{Yu95} J. Yu, Weighted boundary limits of the generalized Kobayashi-Royden metrics on weakly pseudoconvex domains, Trans. Amer. Math. Soc.  {\bf 347} (2) (1995), 587--614.
\end{thebibliography}

\end{document}